\documentclass{amsart}

\usepackage{amssymb}
\usepackage{hyperref}

\newtheorem{theorem}{Theorem}[section]
\newtheorem{lemma}[theorem]{Lemma}
\newtheorem{proposition}[theorem]{Proposition}

\theoremstyle{definition}

\theoremstyle{remark}

\numberwithin{equation}{section}

\usepackage{amstext,esint,enumerate,color}

\providecommand{\nor}[1]{\lVert{#1}\rVert}

\providecommand{\abs}[1]{\lvert{#1}\rvert}
\providecommand{\set}[1]{\{#1\}}
\providecommand{\scal}[2]{\langle{#1},{#2}\rangle}

\providecommand{\det}[1]{\operatorname{det}(#1)}

\providecommand{\dim}[1]{\operatorname{dim}(#1)}
\providecommand{\supp}[1]{\operatorname{supp}(#1)}

\newcommand{\di}[1]{\mathrm{d}#1}
\newcommand{\wh}[1]{\widehat{#1}}

\newcommand{\R}{\mathbb R}
\newcommand{\Cc}{\mathbb C}
\newcommand{\N}{\mathbb N}
\newcommand{\Z}{\mathbb Z}

\newcommand{\hh}{\mathcal H}

\newcommand{\NN}{\mathcal N}

\newcommand{\la}{\lambda}
\newcommand{\eps}{\epsilon}

\newcommand{\T}{\mathcal T} 

\newcommand{\F}{\mathcal F} 

\newcommand{\wR}{\widehat{\R}}
\newcommand{\SO}{\mathrm{SO}(d)}
\newcommand{\SP}{\mathrm{S}^{d-1}}
\newcommand{\ts}{{}^t\!}

\begin{document}

\title[Continuous and discrete frames of Schr\"odingerlets]{Continuous and discrete frames generated by the evolution flow of the Schr\"odinger equation}
\date{November 3, 2016}


\author[G.~S. Alberti]{Giovanni S.~Alberti}
\address{G. S.~Alberti, Seminar for Applied Mathematics, Department of Mathematics, ETH Z\"urich, 8092 Z\"urich, Switzerland.} 
\email{giovanni.alberti@math.ethz.ch}

\author[S. Dahlke]{Stephan Dahlke}
  \address{S. Dahlke, FB12 Mathematik und Informatik, 
    Philipps-Universit\"at Marburg, 
    Hans-Meerwein Stra{\ss}e,
    Lahnberge, 
    35032 Marburg, 
    Germany.} 
   \email{dahlke@mathematik.uni-marburg.de}

\author[F. De Mari]{Filippo~De~Mari}
\address{F. De Mari, Dipartimento di Matematica, Universit\`a di Genova, Via Dodecaneso 35, Genova, Italy.  }
\email{demari@dima.unige.it}

\author[E. De Vito]{Ernesto~De Vito}
\address{E. De Vito, Dipartimento di Matematica, Universit\`a di Genova, Via Dodecaneso 35, Genova, Italy.  }
\email{devito@dima.unige.it}

\author[S. Vigogna]{Stefano~Vigogna}
\address{S. Vigogna, Department of Mathematics, Duke University, 120 Science Drive, 27708 Durham NC, United States.}
\email{stefano@math.duke.edu}

\begin{abstract}
We study a family of coherent states, called Schr\"odingerlets, 
both in the continuous and discrete setting. They are  defined in
terms of the Schr\"odinger equation of a free quantum particle and
some of its invariant transformations.
\end{abstract}


\maketitle

\section{Introduction}

In Quantum Mechanics, the time evolution of a $d$-dimensional free particle is
described by the Schr\"odinger equation
\begin{equation}
  \label{eq:10}
  \begin{cases}
    i \frac{\partial }{\partial t} f(x,t) = -\frac{1}{2\pi}\Delta
    f(x,t) \\
   f(\cdot,0)=f_0  ,
  \end{cases}
\end{equation}
where $\Delta$ is the Laplace operator acting on the ``space'' variable
$x\in\R^d$, and $f_0$ is a
square-integrable function on $\R^d$ describing the state of the
quantum particle at time zero (for the sake of simplicity, the mass is
normalized so that the Laplacian has the simple factor $1/2\pi$).

The aim of this paper is to introduce a new family of coherent states
({\em i.e.}\ a frame) generated by the time evolution unitary operator
defined by the Schr\"odinger equation; following
\cite{reproducingSPI, reproducingSPII, dede13}, its elements are called {\em
  Schr\"odingerlets}. 

Clearly, the time evolution operator $e^{i \frac{t}{2\pi} \Delta}$ is not
enough to generate a frame for $L^2(\R^d)$, hence we need to add 
other unitary transformations. Observe that equation~\eqref{eq:10}
is invariant both with respect to the rotations $R\in \SO$, under the
canonical action
\[ f(x,t)\mapsto f(Rx,t) ,\]
 and with respect to the dilations $a\in \R_+$, under the parabolic  action 
\[ f(x,t)\mapsto a^{\frac{d}{4}}  f(\sqrt{a}x,at),\]
where the factor $a^{\frac{d}{4}} $ ensures that the $L^2$-norm of
$f(\cdot,t)$ is preserved.
Thus, it is natural to consider the group $G=(\R\rtimes \R_+)\times
\SO$, {\em i.e.} the direct product of the identity
component of the one-dimensional affine group and $\SO$, and the
corresponding unitary representation $\pi$ acting  on $L^2(\R^d)$ as
\begin{equation}
\pi(t,a,R)f =a^{-\frac{d}{4}}  e^{i \frac{t}{2\pi} \Delta}
f_{a,R},\label{eq:26}
\end{equation}
where $f_{a,R}(x)= f(a^{-\frac{1}{2}} R^{-1}x)$. It follows that the
solution of~\eqref{eq:10} is given by
\[ f(x,t) =\pi(t,1,\operatorname{I}) f_0(x),\]
and for any rotation $R\in\SO$
\[ f(R x,t) =\pi(t,1,R^{-1}) f_0(x),\]
whereas for any dilation $a\in \R_+$
\[ a^{\frac{d}{4}}  f(\sqrt{a}x,at) = \pi(t,a^{-1},\operatorname{I}) f_0(x).\]

Our goal is to study the properties of the corresponding
family of coherent states $\set{\pi(x)\eta}_{x\in G}$ where $\eta$ is
a suitable ``ground state'', {\em i.e.} an admissible vector. In the
context of signal analysis, this 
amounts to analyzing the {\em voice transform} 
\[ f\mapsto \scal{f}{\pi(\cdot)\eta}\]
as a map from $L^2(\R^d)$ into a suitable Hilbert space of functions on
$G$.  We restrict ourselves to the $L^2$-framework, both in the
continuous and in the discrete
setting. Our main contribution is twofold. First, we show that
$\pi$ is a reproducing representation of $G$ and we
characterize its admissible vectors. This result was already known for
$d=2$ \cite{dede13}, 
and here we extend the proof to 
arbitrary $d$.  Furthermore, 
we construct a discrete Parseval
frame of the form $\set{\pi(x_i)\eta}_{i\in I}$, where  $\set{x_i}_{i\in I}$ is a suitable
sampling of $G$.  

In Section~\ref{sec:2D} we introduce the Schr\"odingerlets in two dimensions
and we discuss the construction of a Parseval frame of two-dimensional Schr\"odingerlets.
The purpose of this dimensionality restriction is twofold. Firstly, it allows to present
the main ideas of this work in a simpler way, so that it may
serve as a good introduction to the more involved general
setting. Secondly, the two-dimensional case is somehow different from
the higher dimensional cases, since when $d=2$  the spherical harmonics on $S^{d-1}$
correspond to the standard Fourier series; thus, a separate
presentation allows to underline the peculiarities of the case $d=2$.

Section~\ref{sec:dD} is devoted to studying the Schr\"odingerlets in any
dimension.  Proposition~\ref{prop:admis} shows that $\pi$ is a
reproducing representation and characterizes its admissible vectors. 
As a consequence, the Schr\"odingerlet voice transform permits
to represent the quantum states as  continuous functions on the parameter
space $\R\times \R_+\times \SO$. Time evolution and rotations correspond to 
translations in the first and third variable, respectively, whereas
dilations give rise to a multi-scale analysis of the original
quantum state. {However, the usual interpretation of
  multi-dimensional wavelets as a combination of (multi-dimensional)
  translations and dilations fails to apply here: translations are
  only one-dimensional, and act on the radial variable in the
  frequency domain, as we shall see below. 
}
{To give a physical  intuition,  let us consider the $3D$ case 
where  $f$, up to a normalization, can be regarded as
the state of a quantum spinless particle and $\eta$ as a probe
state.   If the momentum distribution of $\eta$ is concentrated around
a ball with center $p_0\neq 0$  and radius $\delta$,  clearly the momentum
distribution of $\psi=\pi(0,a,R)\eta$ is concentrated around  a ball
with  center in  $p=Rp_0/a$ and radius $\delta/a$. Hence
$\psi(t)=\pi(t)\psi$ is the free quantum  evolution of the state
$\psi$ and the probe
particle moves in the direction of $p$.  It follows that
$|\scal{f}{\pi(t,a,R)\eta}|^2$ is the transition probability between
$f$ and $\psi(t)$, where $a$ controls the scale change, $R$ the
rotation and $t$ the time. For the classical
directional wavelet transform, $a$ and $R$ play the same role, but the
quantum evolution of $\psi$  is  replaced by a $3D$ translation in the
space. We note that the lack of spatial translations in Schoredingerlets makes hard (if not impossible) to use them for a micro-local analysis, as for example wavefront sets resolution of the signals.} 

The main result of the paper is Theorem~\ref{thm:main}, which  provides sufficient conditions in
order to have a Parseval discrete frame.

We refer to \cite{alanga14, fuhr05} for a general
introduction to coherent states and reproducing formul\ae \ associated
with unitary representations.  Schr\"odingerlets in dimension two were
first introduced in \cite{dede13} and further discussed in
\cite{reproducingSPI,reproducingSPII}, where $G$ is regarded as a
closed subgroup of the symplectic group and $\pi$ is equivalent to the
restriction to $G$ of the metaplectic representation, whose role in
signal analysis has been investigated in a series of papers
\cite{codeno06b,codeno06a,codeno10,cota14,king09}.
 We remark that the representation $\pi$ is reducible and its 
reproducing kernel is not  integrable. Hence, we cannot directly apply
the classical theory of square-integrable representations by Duflo and
Moore \cite{dfmo76}, nor the coorbit space theory developed by
Feichtinger and Gr{\"o}chenig \cite{fegr89a,fegr89b}.

Another construction {of reproducing representations} based on the covariance properties of a free
  quantum particle is given by
the coherent states associated to the isochronous Galilei group (see
\cite[Chapter 8.4.2]{alanga14} and references therein). However,  in
this case, the dilations are not present and the frame does not depend on the
time parameter. Indeed, in order to make the representation square-integrable
it is necessary to reduce the Galilei group by taking the
quotient modulo a group that contains the time translations.

The proof that $\pi$ is a reproducing representation is based on the
general theory developed in \cite{dede13}.
However, in our case, the fact that $\pi$ is the direct sum of a countable family of square-integrable representations $\pi_i$ causes
additional difficulties. A general approach to obtain a discrete frame without
assuming that the kernel is in $L^1(G)$ has been developed in
\cite{ch12,chol09,chol11,fugr07}, but it requires the boundedness of a
suitable convolution operator (see condition (R3) of \cite{ch12}),
which is hard to prove in our setting.  We follow here a different
approach.  Taking into account that $\pi=\bigoplus_i\pi_i$, the
discretization is achieved by a slight generalization of a well known
result on discrete wavelet frames in $L^2(\R)$ \cite[Theorem~1.6,
Chapter~7]{herweis96}, by Schur's orthogonality relations for finite
groups and a technical lemma about Parseval frames
(Lemma~\ref{lem:4}).  Comparing with the approach taken in
\cite{ch12}, we are able to provide only Hilbert frames; we hope to
extend our results in future work and succeed in describing Banach
frames related to the function spaces introduced in
\cite{chol09,chol11,dahlke2014coorbit}. {We also plan to investigate the relation between the decay of the Schr\"odingerlet coefficients and the smoothness properties of the analyzed function.}

\section{The main result in two dimensions}
\label{sec:2D}

We state here the main result of this paper particularized for
two-dimensional signals. In the first part of the section we
introduce the continuous Schr\"odingerlets following
\cite{dahlke2014coorbit}.  

\subsection{The continuous Schr\"odingerlets in 2D}
For $d=2$ we identify the abelian group $\mathrm{SO}(2)$ with the one
dimensional torus $\T=\R/2\pi\Z$ as
\[
\theta\longleftrightarrow
R_\theta=\left[\begin{matrix}
\cos\theta & -\sin\theta \\ \sin\theta &\cos\theta
\end{matrix}\right] .
\]
The group $G$ is $(\R\rtimes \R_+)\times \T$ and its elements
are denoted by $(b,a,\theta)$, writing $b$ instead of the time variable $t$.
In order to better visualize the action of $\pi$ given
by~\eqref{eq:26}, it is worth rewriting it in an equivalent
formulation by means of an intertwining operator $S$ which we shall
now define. We work in the Fourier domain with polar coordinates, and
then perform a Fourier series with respect to the angular
variable.

Below we write $\wR^2$ for the dual space to $\R^2$
and $\mathrm{d}\xi$ and $\mathrm{d}x$ for the corresponding Lebesgue measures.
We denote by $\mathrm{d}\theta$ the Riemannian measure of
$\T$ (so that $\int_{\T} \mathrm{d}\theta=2\pi$). 
We
let $\F\colon L^2(\R^2)\to L^2(\wR^2)$ denote the Fourier transform
given by 
\[
\mathcal F f(\xi)= \int_{\R^2} f(x) e^{-2\pi i x\cdot \xi}\, \di x\qquad\xi\in\wR^2
\]
whenever $f\in L^1(\R^2)\cap L^2(\R^2)$ and $x\cdot \xi$ is the
Euclidean scalar product. 

Define the unitary operator $J\colon  L^2(\wR^2) \to L^2(\wR_+\times \T)$ by
\[ J\hat{f}(\omega,\theta)= \hat{f}(\sqrt{\omega}\cos\theta,\sqrt{\omega}\sin\theta )/\sqrt{2} \qquad \hat{f}\in L^2(\wR^2),\,\omega\in\wR_+,\,\theta\in \T.
\]
The unitarily equivalent representation $(J\F)\pi(J\F)^{-1}$ acting on
$L^2(\wR_+\times \T)$ reads 
\begin{equation}
\label{eq:JF}
 (J\F)\pi(J\F)^{-1}(b,a,\phi)\hat{f}(\omega,\theta)=a^{1/2}e^{-2\pi i
   b \omega}\hat{f}(a\omega,\theta-\phi) \qquad \omega\in \wR_+,\, \theta\in\T  
\end{equation}
for all $(b,a,\phi)\in G$ and $\hat{f}\in L^2(\wR_+\times \T)$.
The action on the radial variable can be described by the representation of $\R\rtimes \R_+$ on $L^2(\wR_+)$ given by
\begin{equation}
\label{eq:W+}
 \wh{W}^+ (b,a)g(\omega)=a^{1/2}e^{-2\pi i
   b\omega}g(a\omega) \quad \omega\in \wR_+,\,(b,a)\in \R\rtimes\R_+ ,\,
 g \in L^2(\wR_+) , 
\end{equation}
which is nothing else than the one-dimensional wavelet representation
in the positive frequency domain. The action on the angular variable is simply given by a rotation $\rho(\phi)z(\theta)=z(\theta-\phi)$ for $z\in L^2(\T)$. Therefore, the action of $\pi$ on two-dimensional functions  should be thought of as a classical one-dimensional wavelet representation on the radial component combined with rotations around the origin.

Consider now the Fourier series with respect to $\theta$ and define
the unitary operator $S\colon L^2(\R^2)\to \bigoplus_{n\in\Z} L^2 (\wR_+)$
by 
\[
(Sf)_n(\omega)= \int_0^{2\pi} (J\F f)(\omega,\theta) e^{-in\theta}
\dfrac{\mathrm{d}\theta}{\sqrt{2\pi}}\qquad
\omega\in\wR_+,\,n\in\Z,\, f\in L^2(\R^2).
\]
From now on, we shall consider the equivalent representation
$\pi'=S\pi S^{-1}$ of $G$ acting on $\bigoplus_{n\in\Z} L^2
(\wR_+)$. In view of \eqref{eq:JF} and \eqref{eq:W+}, the action of
$\pi'$ is given by 
\[
 (\pi'(b,a,\phi)\hat{f})_n=e^{-in\phi}\,(\wh{W}^+
 (b,a)\hat{f}_n)\qquad n\in\Z,\, \hat{f}\in \bigoplus_{n\in\Z} L^2
 (\wR_+),\,(b,a,\phi)\in G.
\]
Denote the character  $\phi\mapsto e^{-in\phi}$ of $\T$ by $\rho_n$;
the representation $\pi'$ can be decomposed as
\begin{equation*}
  \pi'= \bigoplus_{n\in\Z}  \rho_n\wh{W}^+
\end{equation*}
where each component   $\rho_n\wh{W}^+$ acts irreducibly on $L^2
(\wR_+)$. 

It was proven in \cite{reproducingSPII,dahlke2014coorbit} that $\pi'$, and therefore $\pi$, is reproducing, namely
\begin{equation}
\label{eq:reproducing 2D}
 \nor{\hat{f}}_{\bigoplus_{n} L^2 (\wR_+)}^2=\int_G
 |\scal{\pi'(b,a,\phi)\widehat{\eta}}{\hat{f}}|^2\,\mathrm{d} b\frac{\mathrm{d} a}{a^2}\frac{\mathrm{d}\phi}{2\pi}\qquad
 \hat{f}\in \bigoplus_{n\in\Z} L^2 (\wR_+) 
\end{equation}
for some admissible vector $\widehat{\eta}\in \bigoplus_{n\in\Z} L^2 (\wR_+)$. A vector $\widehat{\eta}=(\widehat{\eta}_n)_n\in\bigoplus_{n\in\Z} L^2 (\wR_+)$ is admissible for $\pi'$ if and
only if  
\begin{equation}
  \label{eq:admissible 2D}
  \int_0^{+\infty} \abs{\widehat{\eta}_n(\omega)}^2\frac{\di\omega }{\omega} 
= 1\qquad n\in\Z,
\end{equation}
namely, if and only if each component $\widehat{\eta}_n$ is a
one-dimensional  wavelet \cite{daubechies1992}. A simple way to
construct admissible vectors in $\bigoplus_{n\in\Z} L^2 (\wR_+)$
satisfying \eqref{eq:admissible 2D} is to fix a one-dimensional
wavelet $\widehat{\eta}_0\in L^2 (\wR_+)$ satisfying
\eqref{eq:admissible 2D} and then construct all the other components
$\widehat{\eta}_n$ by dilating $\widehat{\eta}_0$. Since
\eqref{eq:admissible 2D} is invariant under positive dilations, it is
immediately satisfied for all $n$. More precisely, set for all $n\in\Z$
\begin{equation}
\label{eq:eta_n}
\widehat{\eta}_n(\omega)=\widehat{\eta}_0(\alpha_n^{-1} \omega)\qquad\omega\in\wR_+,
\end{equation}
for some weights $\alpha_n>0$ that satisfy $\alpha_0=1$ and $\sum_n
\alpha_n <\infty$. This last condition ensures that the resulting
$\widehat{\eta}$ has finite norm in $\bigoplus_{n\in\Z} L^2 (\wR_+)$,
because 
\[
\nor{\widehat{\eta}}^2=\nor{\widehat{\eta}_0}_{L^2 (\wR_+)}^2\sum_n
\alpha_n.
\]
\subsection{The discrete Schr\"odingerlets in 2D}
We now show how to construct a Parseval frame of $\bigoplus_{n\in\Z} L^2 (\wR_+)$ associated to $\pi'$. Then, by means of the intertwining operator $S$, this frame can be transformed into a Parseval frame of $L^2(\R^2)$ associated to $\pi$. Constructing a Parseval frame corresponds to a discretization of \eqref{eq:reproducing 2D} of the form
\begin{equation*}
\label{eq:frame 2D}
 \nor{\hat{f}}_{\oplus_{n} L^2 (\wR_+)}^2=\sum_{i\in\N} |\scal{\pi'(x_i)\widehat{\eta}}{\hat{f}}|^2\qquad \hat{f}\in \bigoplus_{n\in\Z} L^2 (\wR_+),
\end{equation*}
for suitable choices of the admissible vector $\widehat{\eta}$ and of a sampling $\{x_i\}_{i\in\N}$ of the group $G$.

Our approach is based on the fact that $\pi'$ is the direct sum of wavelet representations {acting on 1D signals}. Thus, it is instructive to look first at the well known one-dimensional case, namely at the representation $\wh{W}^+$ acting on  $L^2 (\wR_+)$. Standard wavelet theory \cite[Thm.~1.1,~Chapter~7]{herweis96} gives that $\{\wh{W}^+ (2^j k,2^j)\,\widehat{\eta}_0:k,j\in\Z\}$ is a Parseval frame for $L^2 (\wR_+)$, namely
\begin{equation*}
\label{eq:frame 1D}
 \nor{\hat{f}}_{L^2 (\wR_+)}^2=\sum_{k,j\in\Z} |\scal{\wh{W}^+ (2^j k,2^j)\,\widehat{\eta}_0}{\hat{f}}|^2\qquad \hat{f}\in L^2 (\wR_+),
\end{equation*}
provided that the conditions
 \begin{subequations}
  \label{eq:assuM1D}\begin{align}  
  &\sum_{j\in\Z}|\widehat{\eta}_0(2^j \omega)|^2 =1,\quad\text{for a.e. $\omega\in\wR_+$,}\\
  & \sum_{j\in\N} \widehat{\eta}_0 (2^j\omega)\overline{\widehat{\eta}_0(2^j(\omega+2\pi m))}=0,\quad\text{for a.e. $\omega\in\wR_+$, $m\in 2\Z+1$}
  \end{align}
 \end{subequations}
hold true. Note that in this case the sampling of the group $\R\rtimes \R_+$ is the discrete set $\{(2^j k,2^j):k,j\in\Z\}$. 

We now generalize this construction to the Schr\"odingerlets. In view
of the above sampling of the affine group, it is natural to consider
the discretization of $G$ given by 
\[
 \{x_{k,j,l}=(2^j k,2^j,2\pi l/L):k,j\in\Z, l=0,\dots,L-1\},
\]
{for some $L\in \Z_+$, where $ \Z_+ := \{1,2,\dots\} $ is the set of positive integers}. Note that the
angles $\phi_l=2\pi l/L$ give a  uniform sampling of $\T$ and form a
finite cyclic subgroup of order $L$. {Let us now discuss suitable
assumptions on the admissible vector $\widehat{\eta}$
so that $\{\pi'(x_{k,j,l})\,\widehat{\eta}:k,j\in\Z, l=0,\dots,L-1\}$ is a
Parseval frame for $\bigoplus_{n\in\Z} L^2 (\wR_+)$. 
}

We first observe that for every $n\in\Z$ it is necessary that each
 $\widehat{\eta}_n\in L^2(\wR^+)$ give rise to a Parseval frame for
 the corresponding space $ L^2(\wR_+)$, {\em i.e}. that each $\widehat{\eta}_n$ satisfies
\eqref{eq:assuM1D} (suitably normalized): 
 \begin{subequations}
 \label{eq:assuM2D-1} \begin{align}
  \label{eq:assuM2D-1-a}
  &\sum_{j\in\Z}|\widehat{\eta}_n(2^j \omega)|^2 =1/L,\quad\text{a.e. $\omega\in\wR_+$, $n\in\Z$},\\
  \label{eq:assuM2D-1-b}
  & \sum_{j\in\N} \widehat{\eta}_n (2^j\omega)\overline{\widehat{\eta}_n(2^j(\omega+2\pi m))}=0,\quad\text{a.e. $\omega\in\wR_+$, $n\in\Z$, $m\in 2\Z+1$.}  
  \end{align}
 \end{subequations}
In the continuous setting, it is necessary and sufficient to assume that each $\widehat{\eta}_n$ is a one-dimensional  wavelet, {\em i.e.}~that \eqref{eq:admissible 2D} holds true for every $n$, in order to have the continuous reproducing formula \eqref{eq:reproducing 2D}. In the discrete case, however, assumptions \eqref{eq:assuM2D-1} are not sufficient, and {we} assume the following conditions {to hold true}:
\begin{subequations}
 \label{eq:assuM2D-2} 
\begin{align}
 \label{eq:assuM2D-2-a}
  &\sum_{j\in\Z}\widehat{\eta}_n(2^j \omega)\overline{\widehat{\eta}_{n+kL}(2^j \omega)} =0,\quad\text{a.e. $\omega\in\wR_+$, $n\in\Z$, $k\in\Z^*$,}\\
  \label{eq:assuM2D-2-b}
  & \sum_{j\in\N}\widehat{\eta}_n(2^j \omega)\overline{\widehat{\eta}_{n+kL}(2^j (\omega+2\pi m))} =0,\,\,\text{a.e. $\omega\in\wR_+$, $n\in\Z$, $k\in\Z^*$, $m\in 2\Z+1$,}  
  \end{align}
 \end{subequations}
where $\Z^*=\Z\setminus\{0\}$.
{When $\eta_n$ is given by \eqref{eq:eta_n}, the above expressions can be
simplified into conditions involving $\widehat{\eta}_0$ and the
weights $\alpha_n$. }

These orthogonality relations do not contain all the cross terms
between $\widehat{\eta}_n$ and $\widehat{\eta}_m$ for $n\neq m$, but only those
corresponding to the cases when $m-n\in L\Z$. The reason for this
simplification can be explained as follows. Two characters $\rho_n$
and $\rho_m$ restricted to the finite subgroup $\{2\pi l/L:
l=0,\dots,L-1\}$ are equivalent if and only if $m-n\in L\Z$. As a
consequence,  all the cross terms corresponding to $m$ and $n$ for
which $m-n\notin L\Z$ are zero by  Schur orthogonality relations for
finite groups.

The following theorem shows that the above conditions are
sufficient {in order to obtain a Parseval frame}.
 \begin{theorem}\label{thm:2D}
Let $\widehat{\eta}\in \bigoplus_{n\in\Z} L^2 (\wR_+)$ be such that \eqref{eq:assuM2D-1} and \eqref{eq:assuM2D-2} hold true, and take $L\in\Z_+$. Then $\{\pi'(x_{k,j,l})\,\widehat{\eta}:k,j\in\Z, l=0,\dots,L-1\}$ is a Parseval frame for $\bigoplus_{n\in\Z} L^2 (\wR_+)$, namely
\begin{equation*}
\label{eq:frame 2D}
 \nor{\hat{f}}_{\oplus_{n} L^2 (\wR_+)}^2=\sum_{k,j,l} |\scal{\pi'(x_{k,j,l})\widehat{\eta}}{\hat{f}}|^2\qquad \hat{f}\in \bigoplus_{n\in\Z} L^2 (\wR_+).
\end{equation*}
 \end{theorem}
The proof is given in Section~\ref{sec:proof} as part of the proof of Theorem~\ref{thm:main}.
Here we just  exhibit  functions $\widehat{\eta}$ satisfying the assumptions. 
Take $\widehat{\eta}_0\in L^2 (\wR_+) $ such that \eqref{eq:assuM2D-1-a} is satisfied for $n=0$ and such that ${\rm supp\,}\widehat{\eta}_0\subseteq [0,2\pi]$. Moreover, choose weights $\alpha_n\in (0,1]$ such that $\alpha_0=1$, $\sum_n \alpha_n <\infty$ and
\begin{equation}
  \label{eq:disjoint-supports-2D}
  |\supp{\widehat{\eta}_0}\cap \alpha_{n}^{-1}
  \alpha_{n+kL}\supp{\widehat{\eta}_0}|= 0 \qquad n\in\Z, k\in\Z^*,
\end{equation}
where $ |\cdot|$ denotes Lebesgue measure.
It is easy to see that the admissible vector  $\widehat{\eta}\in \bigoplus_{n\in\Z} L^2 (\wR_+)$ defined by \eqref{eq:eta_n} satisfies \eqref{eq:assuM2D-1} and \eqref{eq:assuM2D-2}. A simple choice valid for any $L$ is 
$\widehat{\eta}_0=L^{-1}\chi_{[1/2,1]}$ and
\[
 \alpha_n=\begin{cases}
           2^{-2n} & \text{if $n\ge 0$}\\
           2^{2n+1} & \text{if $n< 0$.}
          \end{cases}
\]

We now comment on the role of the number of rotations $L$.
The conclusion of Theorem~\ref{thm:2D} still holds true when $L=1$, namely when no rotations are considered. However, the rotations do play a role in the choice of the admissible vector $\widehat{\eta}$. Indeed, condition \eqref{eq:assuM2D-2}, or \eqref{eq:disjoint-supports-2D} in the case when \eqref{eq:eta_n} holds true, becomes weaker as $L$ increases. More precisely, if $L_2$ is a multiple of $L_1$ and $\widehat{\eta}$ satisfies \eqref{eq:assuM2D-2} with $L=L_1$, then the same equalities hold true with $L=L_2$. Note that this is equivalent to saying that the two  corresponding  discrete subgroups of $\T$ are one contained into the other.

Note that for $L=1$ a simple computation shows that $\nor{\widehat{\eta}}=1$, hence
the frame obtained in Theorem~\ref{thm:2D} is in fact an orthonormal
basis of $\bigoplus_{n\in\Z} L^2 (\wR_+)$ . Indeed, it is a standard  general fact that  a tight frame whose elements have norm (greater than or equal to) one is necessarily an orthonormal basis (see e.g.\
 \cite[Theorem~1.8, Ch. 7]{herweis96}).

\section{The $d$-dimensional case}\label{sec:dD}
\subsection{The continuous setting}\label{continuous}
We define $G=(\R\rtimes \R_+)\times \SO$ as the direct product of the identity
component of the one-dimensional affine group and $\SO$.   Clearly,
the set
\[ H=\set{(0,a,R) \mid a\in\R_{+},\, R\in\SO}\simeq \R_+\times \SO\]
is a closed unimodular subgroup of $G$ and its Haar measure 
is $\di h= a^{-1}\di a\di R$,  and the set
\[ 
\set{(b,1,\mathrm{I})\mid b\in\R}\simeq \R
\]
is a normal abelian closed subgroup of $G$, whose  Haar measure is the
Lebesgue measure~$\di b$. Moreover, $G$ is the
semi-direct product of $\R$ and $H$ with respect to the inner action of
$H$ on $\R$ given by
\[ h[b]=ab \qquad b\in\R,\ h=(a,R)\in H.\]
We set
\begin{equation}
\gamma(h)=\det{\left(b\mapsto h[b]\right)}=a.\label{eq:24}
\end{equation}

The Schr\"odinger representation $\pi$ of $G$ acts on $L^2(\R^d)$ as
\begin{subequations}
  \begin{equation}
    \label{eq:1a}
    \pi(b,a,R) = U(b) V(a,R) \qquad (b,a,R)\in G.
  \end{equation} 
Here $V(a,R)$ is the unitary operator 
\[
V(a,R) f (x)= a^{-\frac{d}{4}} f(a^{-\frac{1}{2}} R^{-1}x) \qquad f\in L^2(\R^d),\,x\in\R^d,
\]
and $b\mapsto U(b)$  is the one-parameter
group  of unitary operators on $L^2(\R^d)$ associated with the Laplacian by  the spectral calculus, namely
  \begin{equation} 
    U(b) = e^{i \frac{b}{2\pi} \Delta}.\label{eq:2a}
  \end{equation}
Thus
  \begin{equation}
    \label{eq:2b}
  \F U(b) \F^{-1} \hat{f}(\xi) =  e^{-2\pi i b \,\xi\cdot\xi}
  \hat{f}(\xi) \qquad \xi\in\wR^d.
  \end{equation}
Setting $\wh{\pi}=\F\pi\F^{-1}$ we get
  \begin{equation}
    \label{pihat}
   \wh{\pi}(b,a,R)\hat{f}(\xi) = a^{\frac{d}{4}} e^{-2\pi i
     b\,\xi\cdot\xi} \hat{f}(a^{\frac{1}{2}} R^{-1}\xi) \qquad
   \hat{f}\in L^2(\wR^d),\,\xi\in\wR^d.
  \end{equation}
\end{subequations}

We now prove that $\pi$ is a reproducing representation.
\begin{proposition}\label{prop:repr}
 The  Schr\"odinger representation $\pi$ of $G$ is a reproducing representation.
\end{proposition}
\begin{proof}
It is enough to prove the result for $\wh{\pi}$, which belongs to the family
of { representations introduced} in
\cite{dede13}, regarding $G$ as semi-direct product of $\R$ and $H$. 
Indeed, $H$ acts on the dual group $\widehat{\R}$ of $\R$ by the
contra-gradient action
\[\ts h[\omega]=a^{-1}\omega\qquad\omega\in\wR ,\,  h=(a,R) \in H.\]
The group $H$ acts on $\R^d$ as well as on the dual space $\wR^d$ by
means of
\[
\begin{array}{rcl}
  h.x & = & a^{\frac12} R x \\
   \ts h. \xi & = & a^{-\frac12} R \xi
\end{array}\qquad  x\in\R^d,\, \xi\in\wR^d,\, h=(a,R)\in H.
\]
We set
\[ \beta(h)=\det{\left(\xi\mapsto \ts h.\xi\right)}=a^{-\frac{d}{2}}.\]
The map 
\begin{equation}
\label{eq:Phi}
\Phi:\wR^d\longrightarrow \wR, \qquad \Phi(\xi)={\xi\cdot\xi}
\end{equation}
is easily seen to
satisfy the following properties:
\begin{enumerate}[i)]
\item $\Phi$ is a smooth map whose gradient is $\nabla\Phi(\xi) = 2\xi$;
\item the set of critical points of $\Phi$ reduces to the origin,
  which is a Lebesgue negligible set, and $\Phi(\wR^d\setminus\set{0})=\wR_+$;
\item $\Phi(\ts h.\xi)=\ts h[\Phi(\xi)]$ for all $\xi\in\wR^d$ and $h\in
  H$;
\item the action of $H$ on $\wR_+$ is transitive, the
  stability subgroup at $1\in\R_+$ is the compact group $\SO$, and $q:(0,+\infty)\to H$,
  $q(\omega)=\omega^{-1}$, is a smooth section, namely
\[ \ts q(\omega)[1]=\omega \qquad \omega\in \R_+;\]
\item $\Phi^{-1}(1)=\SP$, where  $\SP$ is the unit sphere of $\wR^d$ endowed
  with the Riemannian measure $\mathrm{d}s$.
\end{enumerate}
From \eqref{pihat} it is clear that
\begin{subequations}
  \begin{align}
    \label{eq:14a}
    \wh{\pi}(b,h)\hat{f}(\xi)  = \beta(h)^{-\frac12}\, e^{-2\pi i b\Phi(\xi) }\, 
                         \hat{f}(\ts h^{-1}.\xi) ,
\end{align}
\end{subequations}
where $ \xi\in\wR^d$, $\hat{f}\in L^2(\wR^d)$ and $(b,R)\in \R\rtimes(\R_+\times\SO)$, which
shows that $\wh{\pi}$  is the mock-metaplectic representation
associated with the map $\Phi$ (see \cite{dede13}).
Theorem~9 of \cite{dede13} then implies that $\wh{\pi}$ is a reproducing representation.
\end{proof}
{It is worth noting that the above proof works also in the case when $G$ is simply given by $\R\rtimes\R_+$, namely in absence of rotations. Indeed, the only difference is in the stability subgroup at $1$, which  would  be  trivial in this case. Therefore, the corresponding representation is reproducing also without considering rotations.}

We now study the admissible vectors of $\pi$. 
First, we need to recall some elementary facts.

Let $\rho$ be the
regular representation of $\SO$ acting on  $L^2(\SP)$, namely
\[
\rho(R)\varphi(s)=\varphi(R^{-1}s)\qquad s\in\SP, \varphi\in
L^2(\SP), R\in\SO.
\]
There holds that
\begin{equation}
  \label{eq:14}
  L^2(\SP)= \bigoplus_{i\in\N} \hh_i,
\end{equation}
where each $\hh_i$ is the space of spherical harmonics, namely the complex polynomials in $d$ variables,
homogeneous of degree $i$ and harmonic. Here each polynomial is
regarded as a function on  $\SP$, so that $\hh_i$ can be identified as a subspace of
$L^2(\SP)$. For an account of the role of spherical harmonics in the representation theory of the orthogonal groups see \cite{clerc82}.
{Let $ d_i := \dim{\hh_i} $. It is known that
\begin{equation}
  \label{eq:22}
  d_0 = 1,
  \quad d_1=d,
  \quad d_i=\binom{d+i-1}{d-1} - \binom{d+i-3}{d-1},\; i\geq 2.
\end{equation}}
Moreover,
\begin{equation}
  \label{eq:16}
  \rho= \bigoplus_{i\in\N} \rho_i,
\end{equation}
where $\rho_i$ is the restriction of $\rho$ to $\hh_i$.
We  denote by $P_i$ the projection from $L^2(\SP)$ onto $\hh_i$.

If $d>2$, each representation $\rho_i$ is  irreducible,
and two representations $\rho_i$ and $\rho_j$  are inequivalent whenever $i\neq j$ (the
multiplicity of each $\rho_i$ is one).
For $d=2$, every $\hh_i$ with $ i \ge 1 $ has dimension $2$
and each $\rho_i$ is the sum of two inequivalent irreducible one-dimensional representations,
namely $\rho_i^+(\theta)=e^{in\theta}$,
$\rho_i^-(\theta)=e^{-in\theta}$, $\theta\in
\operatorname{SO}(2)\simeq \T$. 
Hence, we still obtain a decomposition into inequivalent irreducible
representations if we just replace the index set $\N$ with $\Z$. 
{For simplicity, we shall proceed using the notation of the case $ d \ge 3$. Except for this minor notational difference, the two-dimensional case described in Section~\ref{sec:2D} is completely covered by the argument given below.}

Recall that the group $\R\rtimes\R_+$ has only two inequivalent
infinite dimensional irreducible representations up to  unitary
equivalence, which we denote by {$\wh{W}^+$ and $\wh{W}^-$} (see e.g. \cite{taylor86}). 
Each of them  acts on
{$L^2(\wR_{+})$ and $L^2(\wR_{-})$, respectively,} as  
\begin{equation}
 \wh{W}^\pm (b,a)\varphi(\omega)=a^{\frac{1}{2}} \varphi(a\omega) e^{-2\pi i
   b\omega} \qquad \omega\in\wR_{\pm},\,
(b,a)\in \R\rtimes\R_+\label{eq:7},
\end{equation}
where $\varphi\in L^2(\wR_{\pm})$.

Now, let $J:L^2(\wR^d) \to L^2( \wR_+ \times\SP)$ be the operator defined by
\begin{equation}
  \label{eq:5}
  J\hat{f}(\omega,s)= \frac{\omega^{\frac{d-2}{4}}}{\sqrt{2}}
  \hat{f}(\sqrt{\omega}\, s)
\qquad \omega\in \wR_+,\, s\in\SP,\,\hat{f}\in L^2(\wR^d).
\end{equation}
We have the following simple lemma.

\begin{lemma}\label{lem:JisUnitary}
The operator $J$ is unitary.  
\end{lemma}
\begin{proof}
If $\hat{f}\in L^2(\wR^d)$, then the changes of variable
$\omega=r^2$ and $\xi=r s$ yield
\begin{align}\label{eq:change}
 \int\limits_{\wR_+\times\SP} \omega^{\frac{d-2}{2}} 
                  \abs{\hat{f} (\sqrt{\omega}s)}^2 \frac{\di\omega\di s }{2}
= \int\limits_{\wR_+\times\SP}  r^{d-1} \abs{\hat{f} (r s)}^2 \di r \di s 
= \int\limits_{\wR^d} \abs{\hat{f} (\xi)}^2 \di\xi.
\end{align}
The inverse of $J$ is given by
\[
 (J^{-1} g)(\xi) = \frac{\sqrt{2}}{(\xi\cdot\xi)^{d-2}}g(\xi\cdot
 \xi, \frac{\xi}{\sqrt{\xi\cdot\xi}}) \qquad \xi\in\wR^d,\, \xi\neq 0 ,\,g\in L^2( \wR_+ \times\SP),
\]
which proves that $J$ is unitary.
\end{proof}
In what follows, we will freely identify
  \begin{align}  \label{identifications}
  L^2(\wR^d) &\simeq
L^2( \wR_+ \times \SP) \nonumber \\
&\simeq L^2( \wR_+)\otimes L^2(\SP)\nonumber\\
&\simeq  \bigoplus_{i\in\N}  L^2(\wR_+)\otimes \hh_i\\
& \simeq\bigoplus_{i\in\N}  L^2(\wR_+,\hh_i).\nonumber
\end{align}

We define the unitary operator $S: L^2(\R^d)\to \bigoplus_{i\in\N}L^2(\wR_+,\hh_i)$ by
\begin{equation}
  \label{eq:4}
  (S f)_i= (\operatorname{Id}\otimes P_i)   (J\F f) \qquad f\in L^2(\R^d) .
\end{equation}

\begin{proposition}\label{prop:admis}
With the above notation,
\begin{equation}
  \label{eq:6}
  S\pi S^{-1} = \bigoplus_{i\in\N}  \wh{W}^+ \otimes \rho_i
\end{equation}
where each component   $\wh{W}^+ \otimes \rho_i$ is irreducible and inequivalent to the others.
A vector ${\eta}\in L^2(\R^d)$ is  admissible for $\pi$ if and
only if  
\begin{equation}
  \label{eq:19b}
  \int_0^{+\infty} \nor{(S{\eta})_i(\omega)}^2_{\hh_i}\frac{\di\omega }{\omega} 
= d_i\qquad i\in\N .
\end{equation}
\end{proposition}
\begin{proof}
{This} proof\footnote{{An alternative proof can  be derived using Proposition 2.23 of \cite{fuhr05}.}} is based on the general theory developed in
\cite{dede13}. We sketch the main steps. 
For any $\omega\in \wR_+$ we denote by $\nu_\omega$ the measure on
$\wR^d$ which is the image measure of
$\omega^{\frac{d-2}{2}} \di s/2$ under the map 
\[ \SP\ni s\mapsto \sqrt{\omega} s \in \wR^d,\]
so that, for all compactly supported continuous functions $\varphi$, we have
\[ \int_{\wR^d} \varphi(\xi) \di\nu_\omega(\xi) = \int_{\SP}
\varphi(\sqrt{\omega}s) \frac{\omega^{\frac{d-2}{2}}}{2}  \di s. 
\] 
The change of variable in spherical coordinates  (as in \eqref{eq:change}) gives
\begin{align*}
  \int_{\wR^d} \varphi(\xi) \di\xi & = \int_0^{+\infty} \left(\int_{\SP}
                                   \varphi(r s) r^{d-1} \di s\right) \di
                                     r \\
& = \int_0^{+\infty} \left(\int_{\SP}
                                   \varphi(\sqrt{\omega} s)
  \frac{\omega^{\frac{d-2}{2}}}{2}  \di s\right) \di \omega \\
 & = \int_0^{+\infty} \left(\int_{\wR^d} \varphi(\xi) \di\nu_\omega(\xi) \right) \di \omega,
\end{align*}
where $r^2=\omega$, so that the disintegration formula
\begin{equation}
  \label{eq:15}
  \di\xi =  \int_0^{+\infty} \nu_\omega\, \di \omega
\end{equation}
holds true. Finally, Weil's formula for quasi-invariant measure on
quotient spaces \cite{fuhr05} reads
\begin{equation}
  \label{eq:21}
  \int_{H} \varphi(a,R) \gamma(a)^{-1} \frac{\di a}{a}\di R = C
  \int_0^{+\infty} \left(\int_{\SO} \varphi(q(\omega)R) \di R\right)\di \omega ,
\end{equation}
for some constant $C$, to be computed. Recalling \eqref{eq:24} and $ q(\omega) = \omega^{-1} $, we obtain $ C = 1 $ since
\[ \int_0^{+\infty} \left(\int_{\SO} \varphi(\omega^{-1}R)\di R\right)
\di \omega= \int_0^{+\infty} \left(
\int_{\SO} \varphi(\omega,R) \di R\right)\frac{\di \omega}{\omega^2}.\] 

Observe that 
\begin{enumerate}[i)]
\item $L^2({\widehat{R}^d},2\nu_1)\simeq L^2(\SP)$;
\item  the ``restriction'' of
  the mock-metaplectic representation $\wh{\pi}$ to the fiber
  $\Phi^{-1}(1)$ and to the stability subgroup $\SO$ is precisely
  $\rho$. Hence,~\eqref{eq:16} provides the
  decomposition of $\rho$ into its irreducibles, all of them with multiplicity 1;
\item up to the normalization factor $1/\sqrt{2}$, the operator $S\F^{-1}$
  coincides with the operator introduced in \cite{dede13}, whose main
feature is that it decomposes $\wh{\pi}$ into its irreducibles,  each of
which  is the canonical representation obtained by inducing the irreducible
  representation of $\R\times \SO$ acting on $\hh_i$ as
\[(b,R)\mapsto  e^{-2\pi i b}\rho_i(R)\]
 from $\R\times \SO$ to $G$.
\end{enumerate}
Theorem 9 of \cite{dede13} shows that {$\eta\in L^2(\R^d)$}
is admissible if and only if, for all $i\in\N$,
\begin{equation*}
  \int_0^{+\infty} \nor{(S{\eta})_i(\omega)}^2_{\hh_i}
  \gamma(q(\omega))\di\omega =
  \frac{\dim{\hh_i}}{C}=d_i .
\end{equation*}
Explicitly, this amounts to
\[
\int_0^{+\infty}  \nor{(S{\eta})_i(\omega)}^2_{\hh_i} \frac{\di\omega}{\omega} =d_i. 
\]
\end{proof}  


{We mentioned above after Proposition~\ref{prop:repr} that the reproducing property holds true also in absence of rotations. The same is true for the construction of the admissible vectors given in Proposition~\ref{prop:admis}. Indeed, it is enough to let $\rho$ be the trivial representation of the identity group and to substitute identity \eqref{eq:14} with
\[
  L^2(\SP)= \bigoplus_{i\in\N} \hh_i,\qquad \hh_i= \mathrm{span}\{h_i\},
\]
where $\{h_i\}_i$ is an orthonormal basis of $L^2(\SP)$, so that expression \eqref{eq:16} still provides the decomposition of $\rho$ into its irreducibles. The rest of the derivation is identical as above. As it is natural, the corresponding conditions \eqref{eq:19b} become stronger in this case, which is the price to pay for the removal of the rotations. This is the continuous counterpart of what observed on the role of rotations for the construction of a Parseval frame (see the comments after Theorem~\ref{thm:2D}).
}

\subsection{A family of admissible vectors}\label{family}

{We now show a procedure to construct all the admissible
  vectors~$\eta \in L^2(\R^d)$ based on the identification provided 
  by~\eqref{identifications}. } {Let ${\eta}\in L^2(\R^d)$
  satisfy~\eqref{eq:19b}.  } 

For any fixed $i\in\N$, we choose an orthonormal basis $\set{e_{i,k}}_{k=1}^{d_i}$ of
$\hh_i$, and define $\varphi_{i,k}:\wR_+\to\Cc$  by 
\[
\varphi_{i,k}(\omega)=\scal{(S{\eta})_i(\omega)}{e_{i,k}}_{\hh_i},
\]
{so that
\begin{align}
   \label{eq:3b}
   (S{\eta})_i =\sum_{k=1}^{d_i} \varphi_{i,k} \otimes e_{i,k}  .
 \end{align}
}
By construction, $\varphi_{i,k}\in L^2(\wR_+)$. Furthermore, 
{condition~\eqref{eq:19b} implies}
\[
{c_{i,k}^2:=}\int_0^{+\infty} \frac{\abs{\varphi_{i,k}(\omega)}^2}{\omega}  \di\omega<+\infty.
\]
If $\varphi_{i,k}\neq 0$, {replacing $e_{i,k}$ with
  $v_{i,k}=c_{i,k}\,e_{i,k} $ in~\eqref{eq:3b}}   we can always assume that 
\begin{subequations}
\begin{equation}
\int_0^{+\infty} \frac{\abs{\varphi_{i,k}(\omega)}^2}{\omega}  \di\omega=1,\label{eq:11}
\end{equation}
{\em i.e.} $\varphi_{i,k}$ is a $1D$-wavelet for $\wh{W}^+$. Hence 
 \begin{align}
   \label{eq:3}
   (S{\eta})_i =\sum_{k=1}^{d_i} \varphi_{i,k} \otimes v_{i,k}  ,
 \end{align}
 where $\set{v_{i,k}}_{k=1}^{d_i}$ is an orthogonal
     family in $\hh_i$  such that
 \begin{equation}
   \label{normalize}
   \sum_{k=1}^{d_i} \nor{v_{i,k}}^2_{\hh_i}  =d_i,
 \end{equation}
and all $\varphi_{i,k}$ satisfy~\eqref{eq:11}.
If for some $k$ the function $\scal{(S{\eta})_i(\cdot)}{e_{i,k}}_{\hh_i}$ is zero, we set
$v_{i,k}=0$ and choose an arbitary  $\varphi_{i,k}$
satisfying~\eqref{eq:11}.

The fact that ${\eta}\in L^2(\R^d)$ implies 
 \begin{equation}
   \label{eltwo}
   \sum_{i=1}^{+\infty} \sum_{k=1}^{d_i} \nor{\varphi_{i,k}}^2_2\,
   \nor{v_{i,k}}^2_{\hh_i} <+\infty.
 \end{equation}
\end{subequations}
Conversely, given a family $(\varphi_{i,k},v_{i,k})_{i\in\N,k=1,\ldots,
    d_i}$ such that 
  \begin{enumerate}[a)]
  \item each $\varphi_{i,k}$ is in $L^2(\wR_+)$ and
    satisfies~\eqref{eq:11},
\item  each family $\set{v_{i,k}}_{k=1}^{d_i}$ is
    orthogonal in $\hh_i$ and satisfies~\eqref{normalize}
    and~\eqref{eltwo},
  \end{enumerate}
then $(\varphi_{i,k},v_{i,k})_{i\in\N,k=1,\ldots,
    d_i}$ defines an admissible vector via~\eqref{eq:3}.
A simple solution is given as follows. Choose a 1D wavelet
$\varphi\in L^2(\wR_+)$. For all $i\in\N$, fix $\alpha_i>0$ and $v_i\in \hh_i$ with
  \[\nor{v_i}^2_{\hh_i}= d_i\]
and
\[ \sum_{i\in\N} \alpha_i d_i <+\infty.\]
Define
\[
\varphi_i(\omega)= \varphi(\alpha_i^{-1}\omega) .
\]
Then, the vector ${\eta}\in L^2(\R^d)$ such that
\[ (S{\eta})_i= \varphi_i\otimes v_i\]
is admissible.

\subsection{Discretization} \label{discr}

The aim of this section is to construct a Parseval frame for
$L^2(\R^d)$ based on a discretization of the 
reproducing representation $\pi$. 

We fix a finite subgroup of  $\SO$ of cardinality $L$
\[ F=\set{R_1,\ldots, R_L},\]
and we choose as  grid points those in the  family
\begin{equation*}
  x_{j,k,\ell} =(2^jk,2^j,R_\ell)\qquad j,k\in\Z,\,\ell=1,\ldots,L.
\end{equation*}
We denote by $\wh{F}$ the set of equivalence classes of   irreducible
(unitary) representations of $F$, and for each equivalence class in
$\wh{F}$ we fix a representative $\chi:F\to
\mathcal{U}(\hh_{\chi})$, where $\hh_{\chi}$ is the Hilbert
space on which  $\chi$ acts and $\mathcal{U}(\hh_{\chi})$ is the
corresponding set of unitary operators. 
The dimension of $\hh_{\chi}$, which is
 always finite, is denoted by $d_{\chi}$.

For each $i\in\N$, the representation $\rho_i$ restricted to $F$
decomposes into its irreducibles
\begin{equation}
  \label{eq:12}
  \hh_i= \bigoplus_{\chi\in\wh{F}}  \hh_{\chi} \otimes \Cc^{m_{i,\chi}}
\qquad {\rho_i}_{|F} = \bigoplus_{\chi\in\wh{F}}  \chi\otimes \operatorname{I}_{m_{i,\chi}},
\end{equation}
where $m_{i,\chi}\in \N$ is the multiplicity of $\chi$
into ${\rho_i}_{|F}$ (with the convention that $\Cc^{0}=\set{0}$ if $m_{i,\chi}=0$, namely when the representation $\chi$ does not enter into the
decomposition). 

\vspace{.2cm}

We remark that in the two-dimensional case the picture is   clearer (see Section~\ref{sec:2D} and  the remarks that follow \eqref{eq:16}). Taking $F=\{2\pi l/L:l=0,\dots,L-1\}$, the set $\wh{F}$ is given by $L$ one-dimensional representations corresponding to the $L$-roots of unity, namely $\wh{F}=\{\chi_l(\cdot)=e^{2\pi i l\cdot/L}:l=0,\dots,L-1\}$. Writing $\hh_k={\rm span}\{e^{ik\cdot}\}$ for $k\in\Z$ (as already observed, the natural index set in 2D is $\Z$), a simple calculation shows that $\rho_k$ corresponds to $\chi_{\bar{k}}$, where $\bar{k}=k\mod L$. Therefore, in the above decomposition one has
\[
 m_{k,\chi_l}=\begin{cases}
             1 & \text{if $k-l\in L\Z$}\\
             0 & \text{otherwise,}
            \end{cases}
\]
or, equivalently, $\hh_k=\hh_{\chi_{\bar{k}}}$.

\vspace{.2cm}

{Back to the case $ d \ge 3 $,} from \eqref{eq:14} and \eqref{eq:12} we finally obtain the decomposition of ${\rho}_{|F}$  into its irreducibles
\begin{equation}
  \label{eq:13}
  L^2(\SP) = \bigoplus_{\chi\in\wh{F}}  \hh_{\chi} \otimes
  \Cc^{m_\chi} \qquad {\rho}_{|F}=
  \bigoplus_{\chi\in\wh{F}}  \chi \otimes \operatorname{I}_{m_\chi},
\end{equation}
where $m_\chi=\sum_{i\in\N}m_{i,\chi}$, the operator $\operatorname{I}_{m_\chi}$ is the
identity  on $\Cc^{m_\chi}$ and
$\Cc^{\infty}=\ell_2(\N)$ if $\sum_{i\in\N}m_{i,\chi}=\infty$. By \eqref{identifications} and \eqref{eq:13},
the following identifications hold true:
\begin{equation} \label{identifications2}
L^2(\wR^d)= \bigoplus_{i\in\N,\chi\in\wh{F}}
L^2(\wR_+, \hh_{\chi}) \otimes \Cc^{m_{i,\chi}} =
\bigoplus_{i\in\N,\chi\in\wh{F}}\bigoplus_{\mu=1}^{m_{i,\chi}} L^2(\wR_+, \hh_{\chi}
\otimes \Cc\set{\eps_\mu}) ,
\end{equation}
where $(\eps_\mu)_{\mu\in\N}$ is the canonical basis of $\ell^2(\N)$ and each
$\Cc^{m_{i,\chi}}$ is regarded as a closed subspace of $\ell^2(\N)$.
According to this decomposition,  we denote by $P_{i,\chi,\mu}$ the orthogonal 
projection from $L^2(\wR^d)$ onto  the closed subspace $L^2(\wR_+, \hh_{\chi}
\otimes \Cc\set{\eps_\mu})$ of $L^2(\wR_+,\hh_i)$.

Next, for each $\chi$, we select an orthogonal {basis}
$w^\chi_{1},\ldots,w^\chi_{d_{\chi}}$  in $\hh_{\chi} $ such that
\begin{equation}
  \label{eq:18}
  \nor{w^\chi_\delta}^2= d_{\chi}\qquad \delta=1,\ldots,d_{\chi}.
\end{equation}
 For each $i\in\N$, we choose $m_{i,\chi}$ vectors in this family
and  denote by
$\Delta_{i,\chi}=(\delta_1,\ldots,\delta_{m_{i,\chi}})$ the
corresponding family of indices {(since any element
  $w^{\chi}_\delta$ can chosen many times, it  can happen that
$\delta_\mu=\delta_{\mu'}$ for some pair of indices)}. We set
\begin{equation}
  \label{eq:17}
 v_{i,\chi,\mu} = w^\chi_{\delta_\mu} \otimes \eps_\mu \qquad
 \mu=1,\ldots, m_{i,\chi},
\end{equation}
where each $v_{i,\chi,\mu}$ is a vector in $\hh_i$ by means
of~\eqref{eq:12}.

Finally, we select $m_{i,\chi}$ functions
$\varphi_{i,\chi,1},\ldots, \varphi_{i,\chi,m_{i,\chi}}\in L^2(\wR_+)$ such that the following conditions hold true:

\begin{enumerate}[a)]
\item the series
\begin{equation}
  \label{eq:25}
  \sum_{i\in\N}\sum_{\chi\in\hat{F}} d_{\chi} \left(\sum_{\mu=1}^{m_{i,\chi}}
  \nor{\varphi_{i,\chi,\mu} }^2_2  \right)<+\infty;
\end{equation}
\item for each $i\in\N$, $\chi\in\wh{F}$ and $\mu=1,\ldots,m_{i,\chi}$
  \begin{subequations}\label{eq:9}
    \begin{equation}
      \label{eq:9b}
      \sum_{j\in\Z} \abs{\varphi_{i,\chi,\mu}(2^j \omega)}^2 =\frac{1}{L} \qquad
      \text{a.e.}\ \omega \in \wR_+,
    \end{equation}
    and for all odd integers $m$
    \begin{equation}
      \label{eq:9c}
      \sum_{j=0}^{+\infty}   \varphi_{i,\chi,\mu}(2^j \omega)
      \overline{\varphi_{i,\chi,\mu}(2^j (\omega+2\pi m))}=0 \qquad
      \text{a.e.}\ \omega \in \wR_+;
    \end{equation}
  \end{subequations}
\item for all $\chi\in\wh{F}$, if there exists $i,i'\in\N$ and
  $\mu=1,\ldots,m_{i,\chi}$, $\mu'=1,\ldots,m_{i',\chi}$ such that
$(i,\mu)\neq(i',\mu')$, but  $w_{\delta_\mu}^{\chi}=w_{\delta_{\mu'}}^{\chi}$ (where $\delta_\mu\in \Delta_{i,\chi}$ and
  $\delta_{\mu'}\in \Delta_{i',\chi}$), then
  \begin{subequations}\label{eq:1a}
    \begin{align}
      \label{eq:1}
       \sum_{j\in\Z} \varphi_{i,\chi,\mu}(2^j
      \omega)\overline{\varphi_{i',\chi,\mu'}(2^j \omega)}=0 \qquad
      \text{a.e.}\ \omega \in \wR_+,
    \end{align}
and for all odd integers $m$
\begin{equation}
  \label{eq:1b}
\sum_{j=0}^{+\infty}   \varphi_{i,\chi,\mu}(2^j \omega)
\overline{\varphi_{i',\chi,\mu'}(2^j (\omega+2\pi m))}=0 \qquad
      \text{a.e.}\ \omega \in \wR_+.
\end{equation} 
  \end{subequations}
\end{enumerate}
Let us comment on the relation between these assumptions and the corresponding ones given in the two-dimensional case.
 Assumption \eqref{eq:25} is simply a restatement of the fact that $\widehat{\eta}$ should have finite norm. Assumptions \eqref{eq:9} and \eqref{eq:1a} correspond to assumptions \eqref{eq:assuM2D-1} and \eqref{eq:assuM2D-2}, respectively. As we have already anticipated when discussing the 2D case,  the condition $(i,\mu)\neq(i',\mu')$ corresponds to $m\neq n$ and $w_{\delta_\mu}^{\chi}=w_{\delta_{\mu'}}^{\chi}$ corresponds to  $m-n\in L\Z$. 

We are now ready to state the main result of this paper.
\begin{theorem}\label{thm:main}
Let ${\eta}\in L^2(\R^d)$ be defined by
\begin{equation}
  \label{eq:32}
  (S{\eta})_i = \sum_{\chi\in\wh{F}} \sum_{\mu=1}^{m_{i,\chi}} \varphi_{i,\chi,\mu} \otimes v_{i,\chi,\mu}.    
\end{equation}
Then the family $\set{ \pi(2^jk,2^j,R_\ell){\eta})}_{j,k\in\Z,l=1,\ldots,L}$ is a Parseval frame for $L^2(\R^d)$.
\end{theorem}
The proof is in~Section~\ref{sec:proof}. We add a few comments.
Since $\sum_{\chi\in\wh{F}} m_{i,\chi} d_{\chi} = d_i$, we have
\[
\sum_{\chi\in\wh{F}} \sum_{\mu=1}^{m_{i,\chi}} \nor{v_{i,\chi,\mu}}_{\hh_i}^2=d_i,
\]
hence~\eqref{eq:25} ensures that~\eqref{eq:32} is well defined
(compare with ~\eqref{eq:3}).

An important result in wavelet theory \cite[Theorem 1.6, Chapter~7]{herweis96} shows
that~\eqref{eq:9b} and~\eqref{eq:9c} are equivalent to the fact that
for each $i\in\N$, $\chi\in\wh{F}$ and $\mu=1,\ldots,m_{i,\chi}$ the family $\set{\wh{W}^+(2^jk,2^j) \sqrt{L}\varphi_{i,\chi,\mu}}_{j,k\in\Z}$ is a
Parseval frame for $L^2(\wR_+)$.  Furthermore, ~\eqref{eq:9b} implies that
\begin{equation}
  \label{eq:20}
  \int_{\wR_+} \frac{\abs{\varphi_{i,\chi,\mu}(\omega)}^2}{\omega} \di\omega 
=\frac{\ln 2}{L},
\end{equation}
so that $\sqrt{ L/\ln 2}\,{\eta}$ is an admissible
vector for $\pi$ by Proposition~\ref{prop:admis}.

We now show that there exist families of $\set{\varphi_{i,\chi,\mu}}$,
satisfying the above conditions. To this end, 
fix a function $\varphi\in L^2(\wR_+)$  supported in $[0,1]$ and such that
\begin{equation}
      \label{eq:99}
      \sum_{j\in\Z} \abs{\varphi(2^j \omega)}^2 =\frac{1}{L} \qquad
      \text{a.e.}\ \omega \in \wR_+.
 \end{equation}
Choose a sequence $\set{\alpha_{i,\chi,\mu}}$ such that $0<\alpha_{i,\chi,\mu}<1$ and
\begin{equation}
\label{eq:31}
\sum_{i\in\N}\sum_{\chi\in\hat{F}} d_{\chi} \sum_{\mu=1}^{m_{i,\chi}} \alpha_{i,\chi,\mu}<+\infty.
\end{equation}
Suppose further that,  for any  $\chi\in\wh{F}$, if there exists $i,i'\in\N$ and
  $\mu=1,\ldots,m_{i,\chi}$, $\mu'=1,\ldots,m_{i',\chi}$ such that
  $(i,\mu)\neq(i',\mu')$ but  $w_{\delta_\mu}^{\chi}=w_{\delta_{\mu'}}^{\chi}$ (where $\delta_\mu\in \Delta_{i,\chi}$ and
  $\delta_{\mu'}\in \Delta_{iÕ,\chi}$), then
\begin{equation}
  \label{eq:29}
|(\supp{\varphi}\cap \alpha_{i,\chi,\mu}^{-1}
  \alpha_{i',\chi,\mu'}\supp{\varphi}|= 0.
\end{equation}
An explicit example is 
\begin{align*}
  \varphi& =\chi_{(1/2,1]} ,\\
  \alpha_{i,\chi,\mu}& = \frac{1}{2^{n_{i,\chi,\mu}}},
\end{align*}
where $(i,\chi,\mu)\mapsto n_{i,\chi,\mu}$ is any bijection from the index set 
\[ \NN=\set{ (i,\chi,\mu)\mid i\in\N,\, \chi\in\wh{F},\, m_{i,\chi}>0,\,
  \mu=1,\ldots,m_{i,\chi}}\]
onto $\N$.

With the above choices, define
\[
\varphi_{i,\chi,\mu} (\omega)= \varphi(\alpha_{i,\chi,\mu}^{-1}\omega)\qquad \omega\in\wR_+.
\]
Now, the sum in \eqref{eq:9c} contains products of the form
\[
\varphi\left(\frac{2^j \omega}{\alpha_{i,\chi,\mu}}\right)
\varphi\left(\frac{2^j \omega}{\alpha_{i,\chi,\mu}}+\frac{2^j 2\pi m}{\alpha_{i,\chi,\mu}}\right).
\]
Since $|2^j 2\pi m/\alpha_{i,\chi,\mu}|>|2^j 2\pi m|>1$ for every odd integer $m$ and every non-negative integer $j$, one of the two factors must always vanish, so that \eqref{eq:9c} holds true.
Similarly,~\eqref{eq:29}
implies~\eqref{eq:1} and ~\eqref{eq:1b}.

\subsection{Proof of Theorem~\ref{thm:main} }\label{sec:proof}

We first prove a technical lemma, which is a variant of a well known result  (see Lemma
1.10 of \cite{herweis96}).

We recall that a family $(\psi_i)_{i\in \N}$ 
in a separable Hilbert
space $\hh$ is a Parseval frame
if one of the following two equivalent conditions is satisfied:
\begin{enumerate}[a)]
\item for all $f\in\hh$ 
\[\sum_{i\in \N} \scal{f}{\psi_i} \psi_i=f ; \] 
\item   for all $f\in\hh$
 \[ \sum_{i\in \N} \abs{{\scal{f}{\psi_i}}}^2=\nor{f}^2,\]
\end{enumerate}
see Theorem~1.7 Chapter 7 of \cite{herweis96}.  Both  series
{converge} unconditionally. For a thorough discussion on frames see e.g.
\cite{ole03,heil11}.

\begin{lemma}\label{lem:3}
Let $(\psi_i)_{i\in\N}$ be a  family of vectors in $\hh$. If there
exists a total subset $\mathcal S$ of $\hh$ such that
\begin{enumerate}[a)]
\item for all $f\in\mathcal S$ the sequence
  $(\scal{f}{\psi_i})_{i\in \N}$ is in $\ell^2(\N)$;
\item for all $f,g\in\mathcal S$
  \begin{equation}
    \label{eq:2}
    \sum_{i\in \N}  \scal{f}{\psi_i}\scal{\psi_i}{g} =\scal{f}{g},
  \end{equation}
\end{enumerate}
then the family $(\psi_i)_{i\in \N}$ is a Parseval frame.
\end{lemma}
\begin{proof}
Define 
\[\mathcal D= \set{f\in \hh\mid \sum_{i\in \N}\abs{\scal{f}{\psi_i}}^2<+\infty }\]
and $V:\mathcal D\to \ell^2(\N)$
\[ Vf= (\scal{f}{\psi_i})_{i\in \N}.\]
By construction, $\mathcal D$ is a linear subspace containing $\mathcal
S$, so that $\mathcal D$ is dense and $V$ is a linear operator. It is
known that $V$ is a closed operator, see Proposition 2.8 of \cite{fuhr05}.    
By~\eqref{eq:2}, the restriction of $V$ to $\mathcal S$ preserves the
scalar product. By linearity, the same property holds on the linear subspace
spanned by $\mathcal S$, which is contained in $\mathcal D$
and  dense in $\hh$ since $\mathcal S$ is total in $\hh$. Then $V$ extends to
a unique isometry $W$ from $\hh$ into $\ell^2(\N)$.  Since $V$ is closed,
then $\mathcal D=\hh$ and $V=W$. By definition of $V$,
 the family $(\psi_i)_{i\in \N}$ is a Parseval frame.
\end{proof}

The following lemma is a variant of a result given in
\cite{fuhr05} in the context of admissible representations,  see Proposition 2.23.
\begin{lemma}\label{lem:4}
Take two countable families $(\hh_j)_{j\in \N}$ and $(\hh'_j)_{j\in \N}$
of separable Hilbert spaces, 
set $\hh=\bigoplus_{j\in \N} \hh_j\otimes \hh'_j$ and, for all $j\in \N$,
denote the canonical
projection by $P_j:\hh\to \hh_j \otimes \hh'_j$.  A family $(\psi_i)_{i\in \N}$ is a Parseval frame for
$\hh$ if and only if  the following two conditions hold true:
\begin{enumerate}[a)]
\item for all $j\in \N$  and all $f\in\hh_j$, $f'\in\hh'_j$
\[  \sum_{i\in \N} |\scal{f\otimes f'}{P_j\psi_i}|^2= \nor{f}^2_{\hh_j} \nor{f'}^2_{\hh'_j} ; \]
\item for all $j,k\in \N$, $j\neq k$ and for all $f\in\hh_j$,
  $f'\in\hh'_j$, $g\in\hh_k$, $g'\in\hh'_k$
\[ \sum_{i\in \N} \scal{f\otimes f'}{P_j\psi_i}\scal {P_k\psi_i}{g\otimes g'}=0.\]
\end{enumerate}
\end{lemma}
\begin{proof}
 Assume that  $(\psi_i)_{i\in \N}$ is a Parseval frame for
$\hh$ and fix $j\in \N$. Given $f\in\hh_j$ and $f'\in\hh'_j$, we have
\[ P^*_j (f\otimes f'_j)= \sum_{i\in \N} \scal{P^*_j  (f\otimes f'_j)}{\psi_i} \psi_i.\]
For all $k\in \N$, $P_k$ is a bounded linear operator, and $P_k P^*_j
=\delta_{jk}P_jP^*_j= \delta_{jk} \operatorname{Id}_{\hh_j\otimes\hh'_j}$. Then
\[
\begin{cases}
 \sum_{i\in \N} \scal{f\otimes f'_j}{P_j \psi_i} P_j\psi_i =f\otimes f'_j& k=j\\
 & \\
 \sum_{i\in \N} \scal{f\otimes f'_j}{P_j \psi_i} P_k\psi_i=0 & k\neq j,
\end{cases}
\]
whence a) and b) easily follow.

Conversely, set 
\[ \mathcal S=\bigcup_{j\in \N} \set{P^*_j(f\otimes f')\mid
  f\in\hh_j,f'\in\hh'_j} ,\]
which is total in $\hh$ by construction. Conditions a) and b) imply
that~\eqref{eq:2} of Lemma~\ref{lem:3} is satisfied, hence,
$(\psi_i)_{i\in \N}$ is a Parseval frame.
\end{proof}

The following result is a restatement of the well known
characterization of wavelet Parseval frames.
For the sake of clarity, we set $\lambda=(j,k)\in \Lambda=\Z^2$ and
$x_\lambda=(2^jk,2^j)\in \R\times\R_+$.
\begin{lemma}
If the family $\set{\varphi_{i,\chi,\mu}}$ in $L^2(\wR_+)$
satisfies~\eqref{eq:9b}, \eqref{eq:9c}, \eqref{eq:1} and \eqref{eq:1b},
then 
\begin{enumerate}[a)]
  \item for each $i\in\N$, $\chi\in\wh{F}$ and $\mu=1,\ldots,m_{i,\chi}$
 \begin{equation}
    \label{eq:27}
    \sum_{\la\in\Lambda}
    \abs{\scal{\varphi}{\wh{W}^+(x_\la)\varphi_{i,\chi,\mu}}_2}^2=\frac{1}{L} \nor{\varphi}_2^2
  \end{equation}
for all $\varphi\in L^2(\wR_+)$;
\item for all $\chi\in\wh{F}$, if there exists $i,i'\in\N$ and
  $\mu=1,\ldots,m_{i,\chi}$, $\mu'=1,\ldots,m_{i',\chi}$ such that $(i,\mu)\neq (i',\mu')$ but
  $w_{\delta_\mu}^{\chi}=w_{\delta_{\mu'}}^{\chi}$ (where $\delta_\mu\in \Delta_{i,\chi}$ and
  $\delta_{\mu'}\in \Delta_{iÕ,\chi}$), then
\begin{equation}\label{eq:30}
\sum_{\la\in\Lambda}
    \scal{\varphi}{\wh{W}^+(x_\la)\varphi_{i,\chi,\mu}}_2 
\scal{\wh{W}^+(x_\la)\varphi_{i',\chi,\mu'}}{\varphi'}_2= 0
\end{equation}
for all $\varphi,\varphi'\in L^2(\wR_+)$.
\end{enumerate}
\end{lemma}
\begin{proof}
The fact that \eqref{eq:27} is equivalent to~\eqref{eq:9b} and~\eqref{eq:9c} is one of the fundamental results at the root of wavelet
frames, see Theorem~1.6 of \cite{herweis96}.   The fact
that~\eqref{eq:1} and~\eqref{eq:1b} imply~\eqref{eq:30} follows by
Lemma~1.18 of \cite{herweis96}, which, by polarization, can be rewritten as
\begin{align*}
   & 2\pi \sum_{\la\in\Lambda}
    \scal{\varphi}{\wh{W}^+(x_\la)\varphi_{i,\chi,\mu}}_2 
\scal{\wh{W}^+(x_\la)\varphi_{i',\chi,\mu'}}{\varphi'}_2 \\
= & \int_{\wR}  \varphi(\omega)\overline{\varphi'(\omega)}\sum_{j\in\Z}
  \varphi_{i',\chi,\mu'}(2^j\omega)\overline{\varphi_{i,\chi,\mu}(2^j\omega)}\di\omega
  \\
+ & \int_{\wR} \overline{\varphi'(\omega)} \sum_{j\in\Z} \sum_{m\in2\Z+1}
  \varphi_{i',\chi,\mu'}(\omega+2^j2\pi m) h_m(2^j\omega)\di\omega ,
\end{align*}   
where
\[
h_m(\omega)= \sum_{n=0}^{+\infty} \varphi_{i',\chi,\mu'}(2^n
\omega) \overline{\varphi_{i,\chi,\mu}(2^n(\omega+2\pi m))} .
\]
Indeed, \eqref{eq:1} implies that the first summand vanishes, whereas \eqref{eq:1b} implies that $h_m$ vanish for all odd integers, hence the second summand is zero.
\end{proof}

\begin{proof}[Proof of Theorem~\ref{thm:main}]
By means of the unitary operator $S$, we can prove the result for the family of vectors
$$ S\pi(x_{\la,\ell}){\eta} \qquad \la \in \Lambda, \quad \ell = 1,\dots,L $$
in the space $ \bigoplus_{i\in\N}L^2(\wR_+,\hh_i) $, which by \eqref{identifications} and \eqref{identifications2} can be identified with
\[ L^2(\wR^d) = \bigoplus_{i\in\N,\chi\in\wh{F}} \bigoplus_{\mu=1}^{m_{i,\chi}}
L^2(\wR_+,\hh_{\chi} \otimes \Cc\set{\eps_\mu}) . \]

We will apply Lemma \ref{lem:4}.
So, let us fix $i,i'\in\N$, $\chi,\chi'\in\wh{F}$ and $ \mu \in \{1,\ldots,m_{i,\chi}\} $, $ \mu' \in \{1,\ldots,m_{i,\chi'}\} $.
Given $\varphi,\varphi'\in L^2(\wR_+)$ and $w\in \hh_{\chi}$, $w'\in\hh_{\chi'}$, we look at the quantity
\begin{align*}
  &A(i,\chi,\mu,i',\chi',\mu') \\
   = &\sum_{\la\in\Lambda}\sum_{\ell=1}^L \scal{\varphi\otimes w\otimes \eps_\mu}{P_{i,\chi,\mu}S\pi(x_{\la,\ell}){\eta}}_2
  \scal{P_{i',\chi',\mu'}S\pi(x_{\la,\ell}) {\eta}}{\varphi'\otimes
  w'\otimes \eps_{\mu'}}_2 .
\end{align*}
Recall that, since $x_{\la,\ell}=(x_\lambda, R_\ell)$, 

\begin{equation*}
  P_{i,\chi,\mu}S\pi(x_{\la,\ell}) {\eta} =  \wh{W}^+(x_{\la})
  \varphi_{i,\chi,\mu} \otimes \chi(R_\ell)  w^{\chi}_{\delta_\mu} \otimes \eps_\mu ,
\end{equation*}
hence we have
\begin{align*}
  A(i,\chi,\mu,i',\chi',\mu') &=  \left(\sum_{\la\in\Lambda}
  \scal{\varphi}{\wh{W}^+(x_\la)\varphi_{i,\chi,\mu}}_2
  \scal{\wh{W}^+(x_\la)\varphi_{i',\chi',\mu'}}{\varphi'}_2\right)\\
&\times  \left(\sum_{\ell=1}^L \scal{w}{\chi(R_\ell) w^{\chi}_{\delta_\mu} }_{\hh_\chi}
  \scal{\chi'(R_\ell) w^{\chi'}_{\delta_{\mu'}}}{w'}_{\hh_{\chi'}} \right) ,
\end{align*} 
where the series are absolutely summable because of~\eqref{eq:27} and the
Cauchy-Schwarz inequality.

From the Schur orthogonality relations applied to the pair of
irreducible representations $\chi,\chi'$ of $F$, we know that
\begin{equation*}
  \frac{1}{L} \sum_{\ell=1}^L \scal{w}{\chi(R_\ell) w^{\chi}_{\delta_\mu}
  }_{\hh_\chi} \scal{\chi'(R_\ell) w^{\chi'}_{\delta_{\mu'}}}{w'}_{\hh_{\chi'}}
  =
  \begin{cases}
    0 & \chi\neq \chi' \\
    \frac{1}{d_{\chi}} \scal{
      w^{\chi}_{\delta_{\mu'}}}{w^{\chi}_{\delta_\mu}}_{\hh_\chi}
    \scal{w}{w'}_{\hh_\chi} & \chi = \chi' .
  \end{cases}
\end{equation*}
Thus, if $\chi\neq\chi'$, we get $A(i,\chi,\mu,i',\chi',\mu')=0$.
From now on assume $\chi=\chi'$, for which
\begin{align*}
  A(i,\chi,\mu,i',\chi,\mu') &= L \left(\sum_{\la\in\Lambda}
  \scal{\varphi}{\wh{W}^+(x_\la)\varphi_{i,\chi,\mu}}_2
  \scal{\wh{W}^+(x_\la)\varphi_{i',\chi,\mu'}}{\varphi'}_2\right)\\
&\times \frac{1}{d_{\chi}} \scal{
      w^{\chi}_{\delta_{\mu'}}}{w^{\chi}_{\delta_\mu}}_{\hh_\chi}
    \scal{w}{w'}_{\hh_\chi} .
\end{align*}
If $(i,\mu)\neq (i',\mu')$ and $\delta_\mu\neq \delta_{\mu'}$, then  $A(i,\chi,\mu,i',\chi,\mu')=0$ since the
family $w^\chi_1,\ldots,w^\chi_{d_\chi}$ is orthogonal. If $(i,\mu)\neq (i',\mu')$ 
but $\delta_\mu= \delta_{\mu'}$,  then by~\eqref{eq:30} it follows that $A(i,\chi,\mu,i',\chi,\mu')=0$.
Finally, if $(i,\mu)=(i',\mu')$, then~\eqref{eq:18} yields
\begin{align*}
  A(i,\chi,\mu,i,\chi,\mu) &= L \left(\sum_{\la\in\Lambda}
  \scal{\varphi}{\wh{W}^+(x_\la)\varphi_{i,\chi,\mu}}_2
  \scal{\wh{W}^+(x_\la)\varphi_{i,\chi,\mu}}{\varphi'}_2\right)
                             \scal{w}{w'}_{\hh_\chi} \\
& = \scal{\varphi}{\varphi'}_2 \scal{w}{w'}_{\hh_\chi} \\
& =\scal{\varphi\otimes w\otimes\eps_\mu}{\varphi'\otimes w'\otimes\eps_\mu}_2  ,
\end{align*}
where the second equality is a consequence of~\eqref{eq:27}.

Summarizing the above results in a single equation,  we obtain
\[
 A(i,\chi,\mu,i',\chi',\mu') =
 \begin{cases}
 \scal{\varphi\otimes w\otimes\eps_\mu}{\varphi'\otimes w'\otimes\eps_\mu}_2 & \text{if $\chi= \chi'$ and $(i,\mu)= (i',\mu')$,}\\
 0 & \text{if $\chi\neq \chi'$ or $(i,\mu)\neq (i',\mu')$.}
 \end{cases}
\]
The conclusion follows from Lemma~\ref{lem:4}.
\end{proof}

\vskip0.2truecm

\section*{Acknowledgments}

G.\ S.\ Alberti was supported by the ERC Advanced Grant Project MULTIMOD-267184. S. Dahlke was supported by Deutsche Forschungsgemeinschaft (DFG), Grant DA 360/19--1. F. De Mari and E.~De Vito  were partially supported by  Progetto PRIN 2010-2011  ``Variet\`a reali e complesse: geometria,
  topologia e analisi armonica''. They are members of the Gruppo Nazionale per l'Analisi Matematica, 
la Probabilit\`a e le loro Applicazioni (GNAMPA) of the Istituto Nazionale di Alta Matematica (INdAM).

The authors would like to thank Fulvio Ricci
for suggesting the extension from the 2D case to  the higher dimensional case.  


\end{document}